\documentclass[12pt]{article}

\usepackage{amsmath,amssymb,amsthm}
\usepackage[utf8]{inputenc}
\usepackage[english]{babel}
\usepackage{tikz}

\newcommand\NN{\mathbb{N}}

\theoremstyle{plain}
\newtheorem{theorem}{Theorem}[section]
\newtheorem{proposition}{Proposition}[section]

\newtheorem{lemma}{Lemma}[section]
\newtheorem{corollary}{Corollary}[section]

\theoremstyle{definition}
\newtheorem{definition}{Definition}[section]
\newtheorem{remark}{Remark}[section]

\newtheorem{example}{Example}[section]
\newtheorem{question}{Question}[section]
\newtheorem{claim}{Claim}

\newcommand{\Z}{\mathbb{Z}} %% Conjunto enteros:       \Z
\newcommand{\CC}{\mathbb{C}} %% Conjunto complejos:     \C
\newcommand{\e}{\varepsilon}
\def\la{\lambda}
\def\al{\alpha}
\def\NN{{\Bbb N}}

\def\e{\varepsilon}
\def\la{\lambda}
\def\d{{\rm d}}
\title{Kreiss bounded and uniformly Kreiss bounded operators}
\author{ A. Bonilla and V. M\"uller
\thanks{The first  author was supported by MINECO and FEDER, Project MTM2016-75963-P. The second author was supported by grant no. 17-27844S of GA CR and RVO:67985840.}}

\begin{document}

\maketitle
\begin{abstract}
If $T$ is a Kreiss bounded operator on a Banach space, then $\|T^n\|=O(n)$.
Forty years ago Shields conjectured  that in  Hilbert spaces,  $\|T^n\| = O(\sqrt{n})$.
A negative answer to this conjecture was given by Spijker, Tracogna and Welfert in 2003.
We improve their result and show that  this conjecture is not true even for uniformly Kreiss bounded operators.
More precisely, for every $\varepsilon>0$ there exists a uniformly Kreiss bounded operator $T$ on a Hilbert space such that
 $\|T^n\|\sim (n+1)^{1-\varepsilon}$ for all $n\in \Bbb N$.
On the other hand, any Kreiss bounded operator on Hilbert spaces satisfies $\|T^n\|=O(\frac{n}{\sqrt{\log n}})$.

We also prove that the residual spectrum of a  Kreiss bounded operator on a reflexive
Banach  space is contained in the open  unit disc,  extending known results for power bounded operators.
As a consequence we obtain examples of mean ergodic Hilbert space operators which are not Kreiss bounded.

2010MSC: 47A10, 47A35

\end{abstract}

\section{Introduction}

Throughout this paper $X$ stands for a complex Banach space, the symbol $B(X)$  denotes the space of all bounded linear  operators acting on $X$, and $X^*$ is the dual of $X$.

\begin{definition}
For an operator $T\in B(X)$ we have three notions of Kreiss boundedness, ordered by strength:
\begin{enumerate}
\item $T$ is \emph{strong Kreiss bounded} if there exists $C>0$ such that
$$
\|(\lambda I-T)^{-k}\| \le \frac{ C}{(|\lambda|-1)^k} \;\; \mbox { for all }  k\in\NN
\mbox{ and } |\lambda |>1;$$
\item $T$ is \emph{uniformly  Kreiss bounded} if there exists $C>0$ such that
$$
\left\|\sum_{k=0} ^{n} \lambda^{-k-1} T^k\right\| \le \frac{ C}{|\lambda|-1} \;\; \mbox { for all } n\in\NN \mbox{ and }|\lambda |>1;
$$

\item $T$ is \emph{ Kreiss bounded } if there exists $C>0$ such that
$$
\|(\lambda I-T)^{-1}\| \le \frac{ C}{|\lambda|-1} \;\; \mbox { for all } |\lambda |>1.
$$
\end{enumerate}
\end{definition}

Given $T\in B(X)$ and $n\ge 0$, we  denote the \emph{Ces\`aro mean} by
$$
M_n(T):=\frac{1}{n+1}\sum_{k=0}^n T^k.
$$

We recall  some definitions concerning the behavior of the sequence of  Ces\`aro means $(M_n(T))$.

\begin{definition}
A linear operator $T$ on a Banach space $X$ is called

\begin{enumerate}
\item \emph{mean ergodic} if $M_n(T)$ converges in the strong operator topology of $X$;
\item \emph{Ces\`{a}ro bounded} if the sequence $(M_n(T))_{n\in \NN}$ is bounded;
\item \emph{absolutely Ces\`aro bounded} if there exists a constant $C > 0$ such that
$$
\frac{1}{N+1} \sum_{j=0}^N \|T^j x\| \leq C \|x\|\;,
$$
for all $x\in X$ and $N\in\NN$.

\end{enumerate}
\end{definition}

An operator $T$ is said to be \emph{power bounded} if    there is a $C>0$ such that $\|T^n\| <C$  for all  $ n$.

The first example of a mean ergodic operator which is not power-bounded was given by
Hille (\cite{Hille}, where $\|T^n\| \sim  n^{1/4}$). An
example of a mean ergodic operator $T$ on  $L^1(\Z)$ with $\limsup_n \|T^n\|/n > 0$ was obtained in
\cite{Kosek} (certainly, $\|T^n x\|/n \rightarrow 0$ for every $x\in L^1(\Z)$).

\begin{remark}
\begin{enumerate}
\item In \cite[Corollary 3.2]{MSZ}, it  is proved that an operator $T$ is  uniformly Kreiss bounded  if and only if there is a constant $C$ such that
\begin{equation}\label{UKB}
\|M_{n}(\lambda T)\| \le C \;\; \mbox{ for all }  n\in\NN \mbox{ and } |\lambda |=1.
\end{equation}

\item In \cite{GZ08}, it was shown that every strong Kreiss bounded operator is   uniformly Kreiss bounded. The converse is not true, see \cite[Section 5]{MSZ}.
Moreover, McCarthy  (see \cite{Mc}, \cite {Shi}) proved that if  $T$ is  strong  Kreiss bounded then $\|T^n\|\le Cn^{1/2}$ (see also \cite[Theorem 2.1]{LN}) and gave also an example of a strong Kreiss bounded operator which is not power bounded.

\item  Denote by
    $$
   M_n^{(2)}(T):= \frac{2}{(n+1)(n+2)}\sum _{j=0}^n (j+1)M_j(T)
   $$
	the second Ces\`aro mean.
It is easy to see that
$$
M_n^{(2)}(T)=\frac{2}{(n+1)(n+2)}\sum_{j=0}^n(n+1-j)T^j.
$$
In \cite {SW}, it was proved that $T$ is Kreiss bounded if and only if there is a constant  $C$ such that
\begin{equation}\label{KB2}
\|M_{n}^{(2)}(\lambda T)\| \le C \;\; \mbox{ for all }  n\in\NN \mbox{ and }|\lambda |=1.
\end{equation}
There exist Kreiss bounded operators which are not Ces\`{a}ro bounded, and conversely \cite{SZ}.

\item An operator $T$ is called M\"obius bounded if its spectrum is contained in the closed unit disc and $\varphi(T)$ is uniformly bounded on the set of the
 automorphism of the unit disc. By \cite{Shi},
$T$ is a M\"obius bounded operator if and only if it is Kreiss bounded.

\item On finite-dimensional spaces, the classes of  Kreiss bounded operators and power bounded operators coincide.

\item By (\ref{UKB}), any absolutely Ces\`{a}ro bounded operator is uniformly Kreiss bounded.
\end{enumerate}
\end{remark}

Let $X$  be the space of all bounded analytic functions $f$   on the unit disc in the complex plane such that the derivative $f'$ belongs to the Hardy space $H^1$, endowed with the norm
$$
\|f\| := \|f\|_{\infty} + \|f'\|_{H^1}\;.
$$
Then  the multiplication operator $M_z$ acting on $X$ is Kreiss bounded but it fails to be power bounded.
Moreover, this operator is not uniformly  Kreiss bounded (see \cite{SW}).

Let $V$ be the Volterra operator acting on $L^p[0,1]$, $1\le p\le \infty$ defined by
$$
(Vf)(t)=\int_0^t f(s) ds\qquad (f\in L^p(0,1)).
$$
Then $I-V$
is uniformly Kreiss bounded. For $p=2$ it is even power bounded (see  \cite{MSZ}).

It is immediate that any power bounded operator is absolutely Ces\`{a}ro bounded. In general, the converse  is not true.

Let $1\le p<\infty$ and let $e_n, n\in\NN$ be the standard basis in $\ell^p(\NN)$.
The following theorem yields examples of absolutely Ces\`{a}ro bounded operators with different behavior on $\ell^p (\NN)$.

\begin{theorem}\cite[Theorem 2.1]{BermBMP} \label{Theorem 2.1}
Let $T$ be the weighted backward shift on $\ell ^p(\NN)$ with $1\leq p<\infty$ defined by $Te_1:=0$ and $Te_k:=w_ke_{k-1}$ for $k>1$. If
 $w_k:=\displaystyle \left( \frac{k}{k-1}\right)^{\alpha}  $ with $0<\alpha <\frac{1}{p}$,
then $T$ is absolutely Ces\`{a}ro bounded on $\ell^p(\NN)$ and is not power bounded.
\end{theorem}

For $p=2$, the adjoint of the operator in Theorem \ref{Theorem 2.1} is uniformly Kreiss bounded, mean ergodic but not absolutely Ces\`{a}ro  bounded.

In \cite{Kornfeld},  Kornfeld and Kosek  constructed for every $\delta \in (0,1)$
a positive mean ergodic operator $T$ on $L^1$ with $\|T^n\|\sim n^{1-\delta}$. By positivity, $T$ is absolutely
Ces\`{a}ro bounded.

Since
\begin{equation}\label{media}
T^n=(n+1)M_n(T)-nM_{n-1}(T),
\end{equation}
any Ces\`aro bounded operator satisfies that  $\| T^n\|=O(n)$.

In the following picture we summarize  the implications among the above definitions.
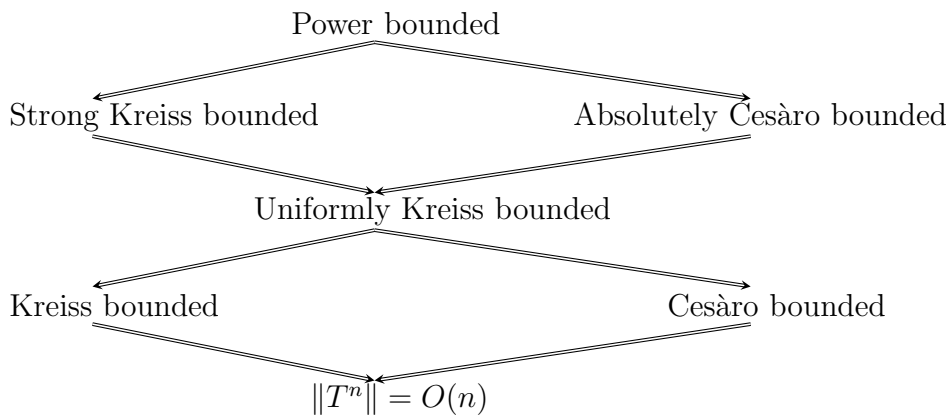
\begin{figure}[h]
%\hspace*{1cm}
\begin{tikzpicture}[scale=0.25,>=stealth]
  \node[right] at (25,15) {Power bounded};
  \node[right] at (10,10) {Strong Kreiss bounded};
  \node[right] at (40,10) {Absolutely Ces\`{a}ro bounded};
  \node[right] at (23,5) {Uniformly Kreiss bounded};
  \node[right] at (10,0) {Kreiss bounded};
  \node[right] at (45,0) {Ces\`{a}ro bounded};
  \node[right] at (26,-5) {\mbox {$\|T^n\|= O(n)$}};
  \draw[double, ->] (30,14) -- (15,11);
  \draw[double, ->] (30,14) -- (50,11);
  \draw[double, ->] (15,9) -- (30,6);
  \draw[double, ->] (50,9) -- (30,6);
   \draw[double, ->] (30,4) -- (15,1);
    \draw[double, ->] (30,4) -- (50,1);
   \draw[double, ->] (15,-1) -- (30,-4);
  \draw[double, ->] (50,-1) -- (30,-4);
 \end{tikzpicture}
\caption{Implications among different definitions related with
Kreiss bounded  and  Cesàro bounded operators on Banach spaces. }
\end{figure}

\section{About the Shields conjecture on Hilbert spaces}

If $T$ is a Kreiss bounded operator in a Banach space, then $\|T^n\|\le Cn$ \cite[ (2.4)]{LN}. By Nevanlinna \cite[Theorem 6]{Ne}, there are Kreiss bounded operators $T$ on Banach spaces with $\|T^n\|\ge C'n$ for some $C'>0$.

In \cite{Shi}, Shields conjectured that any Kreiss bounded Hilbert space operator $T$ satisfies $\|T^n\|=O(\sqrt{n})$.
A negative answer to this conjecture was given in \cite{STW} where it was shown that for every $\e>0$ there exists a Kreiss bounded Hilbert space operator $T$ such that the norms of its powers $\|T^n\|$ grow as fast as $n^{1-\e}$.

In this section we improve this result and construct an operator with similar properties,  which is even uniformly Kreiss bounded.

In the same paper \cite{Shi} Shields mentioned without proof that if $T$ is a Kreiss bounded Hilbert space operator such that the sequence of norms $(\|T^n\|)$ is increasing and there is a
unit vector $x$ such that $\|T^nx||\ge \|T^n\|/2$ for all $n$ then $\|T^n\|=O(\sqrt{n})$ (such properties would satisfy the first natural attempt to disprove the Shields conjecture).
For the sake of completeness we give a proof of this result. We need to prove the following lemma.

\begin{lemma} \label{propos1}
Let $(a_k)_{k=0}^\infty$ be an increasing sequence of non-negative numbers, $B>0$, and let
$\displaystyle\sum _{n=0}^{\infty}a_k^2r^{2k}\le B/(1-r)^2$ for  $0\le r < 1$. Then  $a_n = O(\sqrt n)$.
\end{lemma}
 \begin{proof}

Since $\displaystyle \sum _{n=0}^{\infty}a_k^2r^{2k} \le \frac{B}{(1-r)^2}$,  multiplying both sides by $1-r^2$, we see that
$$
a_0^2 +\sum_{k=1}^{\infty} (a_k^2-a_{k-1}^2) r^{2k} \le \frac{2B}{1-r}.
$$
Now since  $\{a_k\}_{ k\in \Bbb  N}$ is increasing, we have
$$
r^{2n}\Bigl(a_0^2 +\sum_{k=1}^{n} (a_k^2-a_{k-1}^2)\Bigr)  \le a_0^2 +\sum_{k=1}^{n} (a_k^2-a_{k-1}^2) r^{2k} \le \frac{2B}{1-r}.
$$
Set $r =e^{-1/n}$. We conclude that
$$
a_n^2=a_0^2 +\sum_{k=1}^n (a_k^2-a_{k-1}^2) \le B'n
$$
for some constant $B'$.
Thus
$$
a_n = O(\sqrt n).
 $$
\end{proof}

\begin{theorem}
Let $T$ be a Kreiss bounded operator on a  Hilbert space such that $\{\|T^n \|\}_{n=0}^{\infty}$ is increasing and suppose that
there exist a unit vector $x$ and a constant $A$ such that  $\|T^n\|\le A\|T^nx\|$ for all $n$. Then  $\|T^n\| = O(\sqrt n)$.
\end{theorem}

 \begin{proof}

Let $f(z)=\displaystyle \sum _{k=0}^{\infty} T^k z^k$. Since   $T$ is Kreiss bounded we have  $\|f(z)\| \le \frac{ C}{1-|z|}$ for all $|z|<1$.  If $y\in H$ with $\|y\|=1$ then
 $$
 \sum _{k=0}^{\infty}r^{2n}\|T^ny\|^2= \frac{1}{2\pi}\int_{0}^{2\pi}\|f(re^{i\theta})y\|^2 d\theta \le \frac{C^2}{(1-r)^2}.
 $$
Since there exists a unit vector $x$ and   a constant $A$ such that $\|T^n\|\le A\|T^nx\|$  for all $n$, we have
$$
 \sum _{k=0}^{\infty}r^{2n}\|T^n\|^2 \le A^2\sum _{k=0}^{\infty}r^{2n}\|T^nx\|^2 \le \frac{A^2 C^2}{(1-r)^2}.
 $$
Now by Lemma \ref{propos1}, we obtain  $\|T^n\| = O(\sqrt n)$.

 \end{proof}

The Shields conjecture  $\|T^n\|=O(n^{1/2})$ is true for some subclasses of Kreiss bounded operators:

\begin{enumerate}
\item If $T$ is a strong Kreiss bounded operator on a Banach space, then $\|T^n\|=O(n^{1/2})$, see \cite{Mc}.

\item If $T$ is an absolutely Ces\`{a}ro bounded operator on a Hilbert space, then $ \|T^n\|= o(n^{1/2})$ and moreover for all $\varepsilon$ there exist absolutely Ces\`{a}ro bounded  operators on $\ell^2(\mathbb N)$ such
that $ \|T^n\|= O(n^{1/2-\varepsilon})$ \cite{BermBMP}.

\item See \cite{Ghara} for other classes of  Kreiss bounded operators where the Shields conjecture is true.

\end{enumerate}

Now we construct a uniformly Kreiss bounded operator which disproves the Shields conjecture.

\begin{theorem}\label{propos2}
Let $0<\eta<1/2$. Then there exists a constant $c>0$ with the following property:
for each $N\in\NN$ there exists an operator $T_N$ acting on a $2N$-dimensional Hilbert space $H_N$ such that
$$
\|T_N^{2N-1}\|= N^{2\eta},
$$
$$
T_N\hbox{ is a weighted shift satisfying }\|T_N\|= 2^\eta,
$$
$$
\|M_n(T_N)\|\le c \mbox{ for every } n\in \Bbb N.
$$
\end{theorem}

\begin{proof}
Let $H_N$ be the Hilbert space with an orthonormal basis $e_1,\dots,e_{2N}$. Let
$$
w_j=j^\eta\qquad(j=1,\dots,N)
$$
and
$$
w_j=\frac{N^{2\eta}}{(2N-j+1)^\eta}\qquad(j=N+1,\dots,2N).
$$

Consider  the weighted shift  $T_N$ on $H_N$ defined by
$$
T_Ne_j=\frac{w_{j+1}}{w_j} e_{j+1}\qquad(j=1,\dots,2N-1)
$$
and $T_Ne_{2N}=0$.

Note that
$w_1=1$, $w_N=w_{N+1}=N^{\eta}$ and $w_{2N}=N^{2\eta}$.
Then $\|T^{2N-1}\|=\|T^{2N-1}e_1\|=w_{2N}= N^{2\eta}$.

Clearly $\|T_N\|=\max\bigl\{\frac{w_{j+1}}{w_j}:1\le j\le 2N-1\bigr\}=2^\eta$.

Let $n\in\NN$.
We have
$$
\|M_n(T_N)\|=\sup\bigl\{\bigl|\langle M_n(T_N)x,y\rangle\bigr|: x,y\in H_N, \|x\|=\|y\|=1\bigr\}.
$$
Let $x=\sum_{j=1}^{2N} \al_j e_j$, $y=\sum_{j=1}^{2N}\beta_je_j$,
$\|x\|^2=\sum_{j=1}^{2N}|\al_j|^2=1$ and $\|y\|^2=\sum_{j=1}^{2N}|\beta_j|^2=1$.
Let $x=x_1+x_2$ where $x_1=\sum_{n=1}^{N}\al_je_j$ and
$x_2=\sum_{j=N+1}^{2N} \al_je_j$. Similarly, $y=y_1+y_2$, where $y_1=\sum_{j=1}^{N}\beta_je_j$ and
$y_2=\sum_{j=N+1}^{2N}\beta_je_j$.

We have
$$
|\langle M_n(T_N)x,y\rangle|\le A+B+C\eqno(1)
$$
where
$$
A=|\langle M_n(T_N)x_1,y_1\rangle|,
$$
$$
B=|\langle M_n(T_N)x_2,y_2\rangle|
$$
and
$$
C=|\langle M_n(T_N)x_1,y_2\rangle|.
$$

To estimate $A,B $ and $C$
we need two simple lemmas.

\medskip
\noindent{\bf Claim 1.}
There exists a constant $c_1$ such that
$$
\frac{1}{n+1}\sum_{j=1}^{\infty}\sum_{j\le j'\le j+n} \gamma_j\delta_{j'}\frac{j'^\eta}{j^\eta}\le c_1
$$
for all $n$, $\gamma_j,\delta_{j'}\ge 0\quad(j,j'\in\NN)$ with $\sum_j\gamma_j^2=\sum_{j'}\delta_{j'}^2=1$.
\smallskip

\noindent{\bf Proof.}
Let $H$ be the Hilbert space with an orthonormal basis $f_j\quad(j\in\NN)$.
Consider the weighted shift $V\in B(H)$ defined by $Vf_j=\Bigl(\frac{j+1}{
j}\Bigr)^\eta f_{j+1}$. Let
$u=\sum_{j=1}^\infty \gamma_jf_j$ and $v=\sum_{j=1}^\infty \delta_{j'}f_{j'}$ with $\|u\|=\|v\|=1$. By \cite[Corollary 2.4]{BermBMP}, $V$ is uniformly Kreiss bounded. So there exists a constant $c_1$ such that
$$
c_1\ge \|M_n(V)\|\ge
\bigl|\langle M_n(V)u,v\rangle\bigr|=
\frac{1}{n+1}\sum_{j=1}^{\infty}\sum_{j\le j'\le j+n} \gamma_j\delta_{j'}\frac{j'^\eta}{j^\eta}.
$$
\bigskip

\noindent{\bf Claim 2.}
There exists a constant $c_2>0$ such that
$$
\sum_{j=1}^M j^{-2\eta}\le c_2 M^{1-2\eta}
$$
for all $M$.
\smallskip

\noindent{\bf Proof.}
We have
$$
\sum_{j=1}^M j^{-2\eta}\le
1+\int_1^{M}t^{-2\eta} dt=
1+\Bigl[\frac{t^{1-2\eta}}{1-2\eta}\Bigr]_1^M \le
c_2 M^{1-2\eta}
$$
for some constant $c_2>0$ independent of $M$.

\bigskip

\noindent{\bf Continuation of  the proof of Theorem \ref{propos2}:}

We have
$$
A=
\frac{1}{n+1}\sum_{j\le j'\le j+n\atop j'\le N} \al_j\bar\beta_{j'}\frac{w_{j'}}{w_j}
\le
\frac{1}{n+1}\sum_{j\le j'\le j+n\atop j'\le N} |\al_j|\cdot|\beta_{j'}|\frac{{j'^\eta}}{j^\eta}
\le c_1\eqno(2)
$$
by Claim 1.

\bigskip
Similarly,
\begin{align*}
B&=
\frac{1}{n+1}\sum_{N<j\le j'\le j+n, \atop j' \le 2N} \al_j\bar\beta_{j'}\frac{w_{j'}}{w_j}
\le
\frac{1}{n+1}\sum_{N<j\le j'\le j+n, \atop j' \le 2N}|\al_j|\cdot|\beta_{j'}|
\Bigl(\frac{2N-j+1}{2N-j'+1}\Bigr)^\eta
\cr
&\le
\frac{1}{n+1}
\sum_{1\le s'\le s\le \min\{s'+n, N\}} |\al_{2N-s+1}|\cdot|\beta_{2N-s'+1}|\Bigl(\frac{s}{s'}\Bigr)^\eta.
\end{align*}
Setting $\gamma_s=|\alpha_{2N-s+1}|$, $\delta_{s'}=|\beta_{2N-s'+1}|$ in Claim 1, we get
$$
B\le c_1\|x_2\|\cdot\|y_2\|\le c_1.\eqno(3)
$$
\bigskip

To estimate $C$, we distinguish two cases:
\medskip

Let $n+1\ge \frac{N}{2}$. Then
\begin{equation*}\tag{4}
\begin{aligned}
C&=
\frac{1}{n+1}\sum_{j\le j'\le j+n\atop j\le N<j'} \al_j\bar\beta_{j'}\frac{w_{j'}}{w_j}
\le
\frac{2}{N}\Bigl(\sum_{j=1}^{N}\frac{|\al_j|}{w_j}\Bigr)\cdot\Bigl(\sum_{j'>N}|\beta_{j'}|w_{j'}\Bigr)
\cr
&\le
\frac{2}{N}\Bigl(\sum_{j=1}^{N}|\al_j|^2\bigr)^{1/2}\cdot
\Bigl(\sum_{j=1}^{N}w_j^{-2}\bigr)^{1/2}\cdot
\Bigl(\sum_{j'>N}|\beta_{j'}|^2\Bigr)^{1/2}\cdot \Bigl(\sum_{j'>N}w_{j'}^2\Bigr)^{1/2}
\cr
&\le
\frac{2}{N}\|x_1\|\cdot \Bigl(\sum_{j=1}^{N}j^{-2\eta}\Bigr)^{1/2}\cdot\|y_2\|
\cdot N^{2\eta}\cdot
\Bigl(\sum_{s=1}^{N} s^{-2\eta}\Bigr)^{1/2}
\cr
&\le
\frac{2c_2}{N}N^{2\eta}
N^{1-2\eta}= 2c_2.%\mbox{ \eqno(4)}
\end{aligned}
\end{equation*}
\medskip

Let $n+1<\frac{N}{2}$. Then
$$
C=
\frac{1}{n+1}\sum_{j\le j'\le j+n\atop j\le N<j'} \al_j\bar\beta_{j'}\frac{w_{j'}}{w_j}
\le
\frac{2}{n+1} \Bigl(\sum_{j=N-n+1}^{N}|\al_j|\Bigr)\cdot
\Bigl(\sum_{j'=N+1}^{N+n}|\beta_{j'}|\Bigr)
$$
since
\begin{align*}
&\max\Bigl\{\frac{w_{j'}}{w_j}: j\le j'\le j+n, j\le N<j'\Bigr\}
\cr
=&
\frac{w_{N+n}}
{w_{N-n+1}}=\frac{N^{2\eta}}{(N-n+1)^{\eta}(N-n+1)^\eta}
\le
\frac{N^{2\eta}}{(N/2)^{2\eta}}=2^{2\eta}\le 2.
\end{align*}
So
$$
C\le \frac{2}{n+1}\|x_1\|\cdot\sqrt{n}\cdot\|y_2\|\cdot\sqrt{n}\le 2.\eqno(5)
$$

Hence by (1), (2), (3),(4) and (5) we have $\|M_n(T_N)\|\le c$ for all $n$, where $c=2c_1+\max\{2,2c_2\}$.
\end{proof}

\begin{theorem}
Let $\e>0$. Then there exists a uniformly Kreiss bounded operator $T$ on a Hilbert space such that
$\|T^k\|\ge \frac{1}{3}(k+1)^{1-\e}$ for all $k\in\Bbb N$.
\end{theorem}

\begin{proof}
Choose $\eta\in \Bigl(\frac{1-\e}{2},\frac{1}{2}\Bigr)$.
Let
$$
T=\bigoplus_{N=1}^\infty T_N,
$$
where $T_N\quad(N\in\NN)$ are the operators constructed in Theorem \ref{propos2}.

Clearly $T$ is a bounded linear operator, $\|T\|=2^\eta<\sqrt{2}$.

For each $n\in\NN$ we have
$$
\|M_n(T)\|=\sup_N \|M_n(T_N)\|\le c.
$$

Since $T$ is a weighted shift, $\la T$ is unitarily equivalent to $T$ for each $\la\in\CC$, $|\la|=1$ \cite[Corollary 2]{Shi74}. Hence $T$ is uniformly Kreiss bounded.

For each $N\in\NN$ we have
$$
\|T^{2N-1}\|\ge \|T_N^{2N-1}\|=N^{2\eta}>N^{1-\e}\ge \frac{1}{3}(2N)^{1-\e}
$$
and
$$
\|T^{2N}\|\ge\|T_{N+1}^{2N}\|\ge
\frac{\|T_{N+1}^{2N+1}\|}{\|T_{N+1}\|}\ge
\frac{(N+1)^{2\eta}}{2^\eta}>\frac{(N+1)^{1-\e}}{\sqrt{2}}>\frac{1}{3}(2N+1)^{1-\e}.
$$
Hence
$\|T^k\|\ge \frac{1}{3}(k+1)^{1-\e}$ for all $k\in\Bbb N$.

\end{proof}

On the other hand, we prove a small improvement of the  general estimate of norms of powers of Kreiss bounded operators on Hilbert spaces. The proof follows the argument of \cite[Theorem 2.3]{BermBMP} for uniformly Kreiss bounded operators with some necessary  modifications.

\begin{theorem}\label{Hilbert}
Let $T$ be a Kreiss bounded operator in a Hilbert space. Then $\displaystyle \|T^n\|=O\Bigl(\frac{n}{\sqrt{\log n}}\Bigr)$.
\end{theorem}

\begin{proof}
In  \cite {SW} it was proved that $T$ is Kreiss bounded if and only if there is a constant $C'>0$ such that
$$
\|M_{n}^{(2)}(\lambda T)\| \le C' \;\; \mbox{ for }   n=0,1,2, \cdots  \mbox{ and }
|\lambda |=1.
$$
Thus there exists a constant $C>0$ such that
$$
\Bigr\|\sum_{j=0}^{N-1}(N-j)(\lambda T)^j\Bigr\|\le CN^2
$$
for all $\la, |\la|=1$ and $N\ge 1$.

We need several claims.

\begin{claim}
Let $x\in H$, $\|x\|=1$. Then
$$
\sum_{j=0}^{N-1}\|T^jx\|^2\le 16C^2N^2
$$
for all $N\ge 1$.
\end{claim}

\begin{proof}
Consider the normalized Lebesgue measure on the unit circle. We have
\begin{align*}
&N^2\sum_{j=0}^{N-1}\|T^jx\|^2\le
\sum_{j=0}^{2N-1}(2N-j)^2\|T^jx\|^2
\cr
=&\int_{|\la|=1} \Bigl\|\sum_{j=0}^{2N-1} (2N-j)(\la T)^jx\Bigr\|^2 d\lambda \le
16C^2N^4.
\end{align*}
So $\sum_{j=0}^{N-1}\|T^jx\|^2\le 16C^2N^2$.
\end{proof}

\begin{claim}\label{lemma2}
Let $0<M<N$ and $x\in H$, $\|x\|=1$, $T^Nx\ne 0$. Then
$$
\sum_{j=0}^{M-1}\frac{\|T^Nx\|^2}{\|T^{N-j}x\|^2}\le 16C^2M^2.
$$
\end{claim}

\begin{proof}
Set $y=T^Nx$. Since $T^*$ is also Kreiss bounded with the same constant, we have
\begin{align*}
&16C^2M^2\|y\|^2\ge
\sum_{j=0}^{M-1}\|T^{*j}y\|^2
\ge
\sum_{j=0}^{M-1}\Bigl|\Bigl\langle T^{*j}y,\frac{T^{N-j}x}{\|T^{N-j}x\|}\Bigr\rangle\Bigr|^2
\cr
=
&\sum_{j=0}^{M-1}\Bigl|\Bigl\langle y, \frac{T^Nx}{\|T^{N-j}x\|}\Bigr\rangle\Bigr|^2
=\|y\|^2 \sum_{j=0}^{M-1}\frac{\|T^Nx\|^2}{\|T^{N-j}x\|^2}.
\end{align*}
Hence
$$
\sum_{j=0}^{M-1}\frac{\|T^Nx\|^2}{\|T^{N-j}x\|^2}\le 16C^2M^2.
$$
\end{proof}

\begin{claim}\label{lemma3}
Let $N\in\NN$, $x\in H$, $\|x\|=1$ and $T^Nx\ne 0$. Then
$$
\sum_{j=0}^{N-1}\frac{1}{\|T^jx\|}\ge \frac{\sqrt{N}}{4C}.
$$
\end{claim}

\begin{proof}
We have
$$
\sum_{j=1}^{N-1}\|T^jx\|\le\Bigl(\sum_{j=0}^{N-1} \|T^jx\|^2\Bigr)^{1/2}\cdot\sqrt {N}\le
4CN^{3/2}.
$$
Thus
$$
N=\sum_{j=0}^{N-1}\frac{\sqrt{\|T^jx\|}}{\sqrt{\|T^jx\|}}\le
\Bigl(\sum_{j=0}^{N-1}\|T^jx\|\Bigr)^{1/2}\cdot \Bigl(\sum_{j=0}^{N-1}\frac{1}{\|T^jx\|}\Bigr)^{1/2}
$$
and
$$
\sum_{j=0}^{N-1}\frac{1}{\|T^jx\|}\ge
\frac{N^2}{\sum_{j=0}^{N-1}\|T^jx\|}\ge
\frac{N^2}{4CN^{3/2}}=\frac{\sqrt{N}}{4C}.
$$
\end{proof}

\begin{claim}\label{lemma4}
Let $0<M_1<M_2<N$, $\|x\|=1$ and $T^Nx\ne 0$. Then
$$
\sum_{j=M_1}^{M_2-1}\frac{\|T^{N-j}x\|^2}{\|T^Nx\|^2}\ge
\frac{(M_2-M_1)^2}{16C^2M_2^2}.
$$
\end{claim}

\begin{proof}
Let $a_j=\frac{\|T^{N-j}x\|^2}{\|T^Nx\|^2}$. By Claim \ref{lemma2},
$$
\sum_{j=M_1}^{M_2-1}\frac{1}{a_j}\le
\sum_{j=0}^{M_2-1}\frac{1}{a_j}\le 16C^2M_2^2.
$$
We have
$$
M_2-M_1=
\sum_{j=M_1}^{M_2-1}\frac{\sqrt{a_j}}{\sqrt{a_j}}\le
\Bigl(\sum_{j=0}^{N-1}a_j\Bigr)^{1/2}\cdot \Bigl(\sum_{j=0}^{N-1}\frac{1}{a_j}\Bigr)^{1/2}
$$
and
$$
\sum_{j=M_1}^{M_2-1}a_j\ge
(M_2-M_1)^2\cdot \Bigl(\sum_{j=0}^{N-1}\frac{1}{a_j}\Bigr)^{-1}\ge
\frac{(M_2-M_1)^2}{16C^2M_2^2}.
$$
\end{proof}

\medskip

\noindent{\bf Continuation of the proof of Theorem \ref{Hilbert}.}

Let $K\in\NN$ and $2^{K+1}<N\le 2^{K+2}$.
Let $x\in H$, $\|x\|=1$ and $T^Nx\ne 0$.

For $|\la|=1$ let $y_\la=\sum_{j=0}^{2N-1}\frac{(\la T)^jx}{\|T^jx\|}$. Then
$$
\int_{|\la|=1}\|y_\la\|^2 \d\la=2N
$$
and
$$
\int_{|\la|=1}\Bigl\|\sum_{j=0}^{2N-1} (2N-j)(\la T)^jy_\la\Bigr\|^2\d\la\le
16C^2N^4\int_{|\la|=1}\|y_\la\|^2 \d\la\le
32C^2N^5.
$$
On the other hand,
\begin{align*}
&\int_{|\la|=1}\Bigl\|\sum_{j=0}^{2N-1} (2N-j)(\la T)^jy_\la\Bigr\|^2\d\la=
\int_{|\la|=1}\Bigl\|\sum_{j=0}^{2N-1} (2N-j)(\la T)^j
\sum_{r=0}^{2N-1}\frac{(\la T)^rx}{\|T^rx\|}\Bigr\|^2\d\la
\cr
=&\int_{|\la|=1}\Bigl\|\sum_{j=0}^{4N-2}(\la T)^jx\sum_{r=0}^{\min\{j,2N-1\}}\frac{2N-j+r}{\|T^rx\|}\Bigr\|^2\d\la
=
\sum_{j=0}^{4N-2}\|T^jx\|^2\Bigl(\sum_{r=0}^{\min\{j,2N-1\}}\frac{2N-j+r}{\|T^rx\|}\Bigr)^2
\cr
\ge
&\sum_{j=N-2^K}^{N-1}\|T^jx\|^2\Bigl(\sum_{r=0}^j\frac{N}{\|T^rx\|}\Bigr)^2
\ge
N^2\sum_{j=N-2^K}^{N-1}\|T^jx\|^2\Bigl(\frac{\sqrt{N-2^K}}{4C}\Bigr)^2
\ge
\frac{N^3}{32C^2}\sum_{j=N-2^K}^{N-1}\|T^jx\|^2
\cr
\ge
&\frac{N^3}{32C^2}\|T^Nx\|^2\sum_{k=0}^{K-1}\sum_{j=N-2^{k+1}}^{N-2^k-1}\frac{\|T^jx\|^2}{\|T^Nx\|^2}
=
\frac{N^3}{32C^2}\|T^Nx\|^2\sum_{k=0}^{K-1}\sum_{j=2^{k}+1}^{2^{k+1}}
\frac{\|T^{N-j}x\|^2}{\|T^Nx\|^2}
\cr
\ge
&\frac{N^3}{32C^2}\|T^Nx\|^2\sum_{k=0}^{K-1}\frac{2^{2k}}{16C^2\cdot 2^{2k+2}}=
\frac{N^3}{2^{11}C^4}\|T^Nx\|^2 K.
\end{align*}
Thus we have
$$
\|T^Nx\|^2\le\frac{2^{16}C^6N^2}{K}\le\frac{2^{16}C^6N^2}{\log_2 N-2}.
$$
Hence  $\|T^N\|=O(\frac{N}{\sqrt{\log N}})$.

\end{proof}

In the next diagram we show graphically the implications among various definitions related with
Kreiss boundedness on Hilbert spaces and corresponding known estimates for  the growth of  $\|T^n\|$.
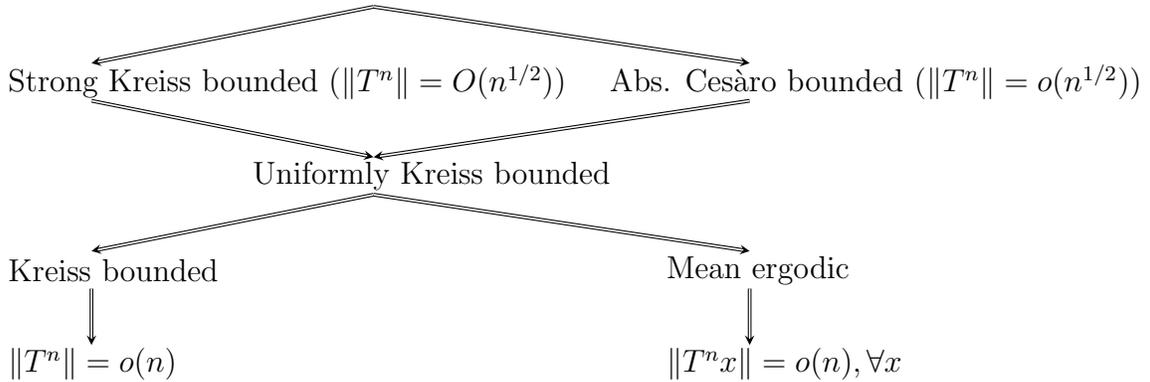
\begin{figure}[h]
%\hspace*{1cm}
\begin{tikzpicture}[scale=0.25,>=stealth]
  \node[right] at (10,10) {Strong Kreiss bounded  \mbox {($\|T^n\|= O(n^{1/2})$)}};
 \node[right] at (42,10) {Abs. Ces\`{a}ro bounded  \mbox {($\|T^n\|= o(n^{1/2})$)}};
  \node[right] at (23,5) {Uniformly Kreiss bounded};
  \node[right] at (10,0) {Kreiss bounded};
  \node[right] at (45,0) {Mean ergodic};
 \node[right] at (10,-5) {\mbox {$\|T^n\|= o(n)$}};
  \node[right] at (45,-5) {\mbox {$\|T^n x\|= o(n), \forall x$}};
  \draw[double, ->] (30,14) -- (15,11);
  \draw[double, ->] (30,14) -- (50,11);
  \draw[double, ->] (15,9) -- (30,6);
  \draw[double, ->] (50,9) -- (30,6);
   \draw[double, ->] (30,4) -- (15,1);
   \draw[double, ->] (30,4) -- (50,1);
    \draw[double, ->] (15,-1) -- (15,-4);
  \draw[double, ->] (50,-1) -- (50,-4);
 \end{tikzpicture}
\caption{Implications among different definitions related with
Kreiss bounded on Hilbert spaces. }
\end{figure}

\begin {example} Derriennic \cite{D} gave an example of a mean ergodic operator $T$ on
a real Hilbert space for which $n^{-1}\|T^n\|$ does not converge to zero, and such that $T^*$
is not mean ergodic,  only weakly mean ergodic (i.e., the Ces\`aro means converge weakly).  In \cite[Example 3.1]{TZ} it was shown that the operator
 \[
{T}:= \left(
\begin{array}{lc}
 B& B-I  \\
0 & B  \\
\end{array}
\right)
\]
acting on the Hilbert spaces $l^2(\mathbb{N}) \oplus l^2(\mathbb{N})$, where  $B$ is the backward shift in  $l^2(\mathbb{N})$, is mean ergodic and $n^{-1}\|{T}^n\|\ge 2$. As a consequence of Theorem \ref{Hilbert}, ${T}$ is an example of a mean ergodic operator acting on a Hilbert spaces, which is not Kreiss bounded.

\end{example}

\bigskip

By \cite[Remark 3.1]{AB}, in Banach spaces there exists a Kreiss bounded operator such that
$$
\displaystyle \lim_{n\to\infty}\|M_{n+1}(T)- M_{n}(T)\|\neq 0.
$$
 However, in Hilbert spaces this is not possible.

\begin{theorem}
If  $T$ is a Kreiss bounded operator on a Hilbert space then
$$
\displaystyle \lim_{n\to\infty}\|M_{n+1}(T)- M_{n}(T)\|=0.
$$
\end{theorem}
\begin{proof}
We have
$$
\frac{n+2}{n+1}M_{n+1}(T)- M_{n}(T) = \frac{1}{n+1} T^{n+1}.
$$
So
 $$
M_{n+1}(T)- M_{n}(T) = \frac{1}{n+1} T^{n+1} - \frac{1}{n+1}M_{n+1}(T).
$$
Now if $T$ is a Kreiss bounded operator on a  Hilbert space, then we have $\|T^n\|=o(n)$ by Theorem \ref{Hilbert}   and $\|M_n(T)\| = O(\log(n+2))$  (see, \cite[Theorem 6.1, 6.2] {SW}). Thus
$$
\displaystyle \lim_{n\to\infty}\|M_{n+1}(T)- M_{n}(T)\|=0.
$$
\end{proof}
\section{On residual spectrum of  Kreiss bounded operators}

The following characterization of ergodic operators was proved in \cite{K}.

\begin{theorem}\cite[Theorem 2.1.3, page 73]{K}
An operator $T$ in a Banach space $X$ is mean ergodic if and only if it is Ces\`aro bounded, $\|T^nx\| = o(n)$ for all $x\in X$ and
$$
X= N(I-T) \oplus\overline{R(I-T)}
$$
\end{theorem}

Recall that the residual spectrum $\sigma_R(T)$ of an operator $T\in B(X)$ is the set of all $\la\in\CC$ such that $T-\la$ is injective and $\overline{R(T-\la)}\ne X$.

\begin{corollary} If $T$ is mean ergodic in a Banach space, then $1\notin \sigma_R(T)$
\end{corollary}

\begin{corollary}\cite[Corollary 2.5]{BermBMP}, \cite[Corollary 2.7]{BermBMP}
 Let $T$ be a  bounded operator on a Banach  space $X$. If
\begin{enumerate}
\item either $T$ is a uniformly Kreiss bounded  operator on a Hilbert space,

\item or $T$ is an absolutely Ces\`aro  bounded  operator on a reflexive Banach  space,

\end{enumerate}
then $\lambda T$ is mean ergodic for all $\la, |\la|=1$. Consequently, $\sigma_R(T)\subset{\mathbb D}$.
\end{corollary}

The next result generalizes the above observation as well as the results of \cite{MK}, \cite{GZ08}  for power bounded operators.

\begin{theorem}\label{residualspectrum}
If $T$ is a Kreiss bounded operator on a reflexive Banach space then $\sigma_R(T)\subset D$.
\end{theorem}

\begin{proof}
If $T$ is a Kreiss bounded operator on a Banach space $X$, then $\|M_n^{(2)}(\lambda T)\|$ are uniformly bounded and $\|M_n(\lambda T)\| = O(log(n+2))$ for all $ \lambda \in \mathbb{T}$ (see, \cite[Theorem 6.1, 6.2] {SW}).

Now, since $X$ is a reflexive Banach space, $ M_n^{(2)}(\lambda T)$ converge strogly in $X = N(I-\lambda T) \oplus \overline{R(I-\lambda T)}$ (see, \cite[Theorem 2.1]{Su3}).

So for all $ \lambda \in \mathbb{T}$, we have $X = N(\overline{\lambda}I- T)\oplus \overline{R(\overline{\lambda}I-T)}$. Hence if $\overline{\lambda} \notin \sigma_p(T)$, then $X =  \overline{R(\overline{\lambda}I-T)}$ and
$\overline{\lambda} \notin \sigma_R(T)$.

Thus $ \sigma_R(T)\subset D$.

\end{proof}

The condition on the residual spectrum is optimal. The forward shift in $l^{2}(\mathbb{N})$ is a  power bounded operator with  residual spectrum equal to the open unit disc.

\begin{example}
There exists  a power  bounded operator $T$ on $c_0(\mathbb{N})$ such that $1\in \sigma_R(T)$.
\end{example}
\begin{proof}
The operator $T: c_0(\mathbb{N})\rightarrow c_0(\mathbb{N})$ defined by
$$
T(a_1 , a_2 , a_3 , \cdots )=(a_1, a_1 , a_2 , a_3 , \cdots )
$$
is power-bounded and $1\in \sigma_R(T)$.
\end{proof}

\begin{example}
There exists  a Kreiss bounded operator  $T$ on a non-reflexive Banach space, which is not power bounded and  $1\in \sigma_R(T)$.
\end{example}
\begin{proof}
Let $X$ denote the Banach space of analytic functions $f$ in the open unit disc and continuous on the boundary, such that $f'$ belongs to the Hardy space $H^1$, equipped with the norm
$$
\|f\|:= \|f\|_{\infty}+ \|f'\|_{1}
$$
If $M_z$ denotes the multiplication operator and $T= \frac{1}{2}(I+M_z)$, then $T$ is a Kreiss bounded operator, see \cite[Example 4]{Ne}.

Moreover, $N(I-T)=\{0\}$ and $R(I-T)$ is not dense because every function in this closure necessarily verifies that $f(1)=0$. Thus  $1\in \sigma_R(T)$.
 \end{proof}

\begin{proposition}\label{ergces}
There exists a Ces\`aro bounded operator $T$ on a Hilbert space  which is mean ergodic,
$N(T+I)\ne\{0\}$ and $N(T^*+I)=\{0\}$.
\end{proposition}

\begin{proof}
Let $H$ be the Hilbert space with an orthonormal basis $e_j \quad(j=0,1,\dots)$.
Let $\e_j=2^{-j}, c_j=1-\e_j^2\quad(j\ge 1)$.
Let
\[
T=-\left(
\begin{array}{lcccr}
1 & \e_{1} & \e_{2}&\e_3& \cdots \\
0 & c_{1} & 0& 0 & \cdots \\
0 & 0 & c_{2} & 0 &\cdots \\
0 & 0 & 0 & c_{3} & \cdots \\
&&\cdots&&
\end{array}
\right)
\]

Clearly $Te_0=-e_0$, so $N(T+I)\ne\{0\}$. We have $(T+I)e_j=(1-c_j)e_j-\e_je_0=\e_j^2e_j-\e_j e_0$. So
$(T+I)(-\e_j^{-1}e_j) =e_0-\e_je_j\to e_0$.
Thus $e_0\in\overline{R(T+I)}$ and it is easy to see that $\overline{R(T+I)}=H$.
Hence $N(I+T^*)=\{0\}$.

For $n\ge 1$ we have
$$
((T)^n)_{j,j}=(-c_j)^n\qquad(j\ge 0),
$$
$$
((T)^n)_{0,j}=(-1)^n \e_j(1+c_j+c_j^2+c_j^{n-1})=
(-1)^n\e_j\frac{1-c_j^n}{1-c_j}\qquad(j\ge 1)
$$
and $((T)^n)_{i,j}=0$ otherwise.

So
\[
T^n=(-1)^n\left(
\begin{array}{lccr}
1 & \e_{1}(1+c_1+\cdots+c_1^{n-1}) & \e_{2}(1+\cdots+c_2^{n-1})& \cdots \\
0 & c_{1}^n & 0& \cdots \\
0 & 0 & c_{2}^n & \cdots \\
&\cdots&& \cdots
\end{array}
\right)
\]
We show that $n^{-1}\|T^n\|\to 0$.
Let $\delta>0$. Find $j_0$ such that $\e_{j_0}<\delta$ and $n_0$ satisfying
$c_{j_0}^{n_0\delta}<\delta$ and $2n_0^{-1}<\delta$.

Let $n\ge n_0$. Clearly $|(T^n)_{j,j}|\le 1$ for all $j\ge 0$.

If $j\ge j_0$ then
$$
|(T^n)_{0,j}|=\e_j(1+c_j+\cdots+c_j^{n-1})\le\e_j n=2^{-j}n<\frac{\delta n}{2^{j-j_0}}.
$$
If $1\le j<j_0$ then
$$
|(T^n)_{0,j}|=\e_j(1+c_j+\cdots+c_j^{n-1})=
\e_j\sum_{0\le i<n\delta}c_j^i+\e_j\sum_{n\delta\le i\le n-1}c_j^i\le
\e_j(n\delta+1)+\e_j n\delta.
$$
Hence
\begin{align*}
\|T^n\|&\le \sup_j |(T^n)_{j,j}|+\sum_{j=1}^\infty |(T^n)_{0,j}|
\cr
&\le
1+\delta n\sum_{j=j_0}^\infty 2^{-j+j_0}+\sum_{j=1}^{j_0-1}\bigl(\e_j(n\delta+1)+\e_j n\delta\bigr)\le
1+2\delta n+\sum_{j=1}^{j_0-1} 2^{-j}(2n\delta+1)\le 5\delta n.
\end{align*}
Hence $n^{-1}\|T^n\|\to 0$.

We show that $T$ is Ces\`aro bounded.

Let $M_{2k}(T)=(2k+1)^{-1}(I+T+T^2+\cdots+T^{2k})$. It is easy to see that
$$
|(M_{2k}(T))_{j,j}|\le 1.
$$
For each $k\ge 1$ we have
\begin{align*}
&\bigr|(M_{2k}(T))_{0,j}\bigl|=
\frac{\e_j}{(2k+1)(1-c_j)}\Bigl|-(1-c_j)+(1-c_j^2)-\cdots+(1-c_j^{2k})\Bigr|
\cr
=&\frac{\e_j}{(2k+1)(1-c_j)}\bigl|c_j-c_j^2+\cdots-c_j^{2k}\bigr|=
\frac{\e_j}{2k+1}\bigl(c_j+c_j^3+\cdots+c_j^{2k-1}\bigr)\le\frac{\e_j}{2}.
\end{align*}
So
$$
\|M_{2k}(T)\|\le 1+\sum_{j=1}^\infty\frac{\e_j}{2}=\frac{3}{2}.
$$
Since
$$
M_{2k+1}(T)=\frac{2k+1}{2k+2}M_{2k}(T)+\frac{T^{2k+1}}{2k+2},
$$
$T$ is Ces\`{a}ro bounded. Moreover,  as $n^{-1}\|T^n\|\to 0$,  $T$ is  mean  ergodic.
\end{proof}

\begin{corollary}\label{meanergodic} There exists a mean ergodic  operator $T$ on a Hilbert space such that  $\sigma_R(T)\cap \partial \mathbb {D} \ne \emptyset$.
\end{corollary}

\begin{example}
By Theorem \ref {residualspectrum}, the operator of Corollary \ref{meanergodic} is another example of a mean ergodic operator on a Hilbert space, which is not Kreiss bounded.
\end{example}

\section{Questions}

As consequence of results of this paper we gave two examples in Hilbert spaces of mean ergodic operators which are not Kreiss bounded.
By \cite[Corollary 2.5]{BermBMP}, all uniformly Kreiss bounded operators on Hilbert spaces are mean ergodic.
 However, the following problem is open:
 \begin{question}
 Does there exist a Kreiss bounded operator on a Hilbert space which is not mean ergodic?
\end{question}

Observe that by Theorem \ref{Hilbert}, the above problem is equivalent to the question whether there exists a Kreiss bounded Hilbert space operator which is not Ces\`aro bounded.

\bigskip

If the answer of the above question is positive then such an operator is not uniformly Kreiss bounded. If the answer is negative then it is natural  to ask
\begin{question}
Does there exist a Kreiss bounded Hilbert space operator which is not uniformly Kreiss bounded?
\end{question}

Another open question is whether it is possible to generalize the Jacobs-de Leeuw-Glicksberg theorem for uniformly Kreiss bounded operators.

\begin{question}
Let $X$ be a reflexive Banach space and $T\in B(X)$ a uniformly Kreiss bounded operator such that $n^{-1}T^nx\to 0$ for all $x\in X$. Is it true that $X$ can be decomposed as
$$
X=\bigvee_{|\lambda|=1}N(T-\la) \oplus \bigcap_{|\la|=1}\overline{R(T-\la)}\ \ ?
$$
\end{question}

Clearly for Hilbert space operators the condition $n^{-1}T^nx\to 0$ is satisfied automatically.

It is an open question due to Aleman and Suciu \cite{AS}, p.279 whether each uniformly Kreiss bounded operator $T$ on a Banach space $X$ satisfies the condition $\|n^{-1}T^n\|\to 0$.

\small

\smallskip
\noindent A. Bonilla\\
\textit{  Departamento de An\'alisis Matem\'atico, Universidad de la Laguna,
38271, La Laguna \\(Tenerife), Spain.}\\
\textit{  e-mail:} abonilla@ull.es

\smallskip
\noindent V. M\"uller\\
Mathematical Institute, Czech Academy of Sciences, Zitn\'a 25,
115 67 Prague 1, Czech Republic.\\
{\it e-mail}: muller@math.cas.cz

\end{document}